\documentclass[ECP,preprint]{ejpecp} 

\SHORTTITLE{Dominating occupancy processes}

\TITLE{Dominating occupancy processes by the independent site approximation 
    } 

\AUTHORS{%
 Ross McVinish\footnote{University of Queensland, Australia.
    \EMAIL{r.mcvinish@uq.edu.au}}}


\AMSSUBJ{60J10; 60J27; 60K35} 

\SUBMITTED{December xx, 2021} 
\ACCEPTED{} 

\VOLUME{0}
\YEAR{2020}
\PAPERNUM{0}
\DOI{10.1214/YY-TN}

\ABSTRACT{Occupancy processes are a broad class of discrete time Markov chains on $\{0,1\}^{n}$ encompassing models from ecology and epidemiology. 
This model is compared to a collection of $n$ independent Markov chains on $\{0,1\}$, which we call the independent site model. We establish conditions under 
which an occupancy process is smaller in the lower orthant order than the independent site model. An analogous result for spin systems follows by a limiting argument.}

\begin{document}



\section{Introduction}

Occupancy processes \cite{HMP:20} are a class of discrete time Markov chains on $ \{0,1\}^{n} $. This class encompasses  models from diverse areas 
including Hanski's incidence function model \cite{Hanski:94}, which is one of the most important models in metapopulation ecology, contact-based epidemic 
spreading processes \cite{GABHMM:10} and dynamic random graph models \cite{HFX:10}. Furthermore, it was shown in \cite{MH:21} that occupancy 
processes are natural time discretisation for finite spin systems such as contact process, voter model and Ising model \cite{Liggett:05}. 

We define the occupancy process $ (X,\, t\in \mathbb{N}) $ as a discrete time Markov chain on $ \{0,1\}^{n} $ where, conditional on $ X_{t}$, the 
$ X_{i,t+1}, \, i=1,\ldots,n, $ are independent with transition probabilities
\begin{equation}
   \mathbb{P}\left(X_{i,t+1} = 1 \ \middle| \ X_{t} \right)  =  C_{i}(X_{t})\left(1-X_{i,t}\right) +  S_{i}(X_{t})X_{i,t}, \label{Eq:model1}
\end{equation}
where the functions $ C_{i}: \{0,1\}^{n} \to [0,1] $ and $ S_{i}: \{0,1\}^{n} \to [0,1] $ are called the colonisation and survival functions of site $ i $ 
in reference to metapopulation modelling. We interpret $ X_{i,t} =1  $ as  site $ i $ supporting a population at time $ t $, and $ X_{i,t}=0 $ as site $ i $
 not supporting a population at time $ t $.  If site $i$ does not support a population at time $ t $, then the site will be colonised at time $ t+1 $ with 
 probability $ C_{i}(X_{t})$. Similarly, if site $i$ supports a population at time $ t$, the population will survive to time $ t+1 $ with probability $ S_{i}(X_{t})$. 

Although the occupancy process is a finite state Markov chain, the size of the state space usually renders standard analysis intractable. Instead a variety of 
approximations are employed to understand the process's behaviour. Provided the colonisation and survival functions can be extended from $\{0,1\}^{n} $ 
to  $ [0,1]^{n}$, a natural approximation of (\ref{Eq:model1}) is the deterministic process
\begin{equation}
p_{i,t+1} = C_{i}(p_{t})\left(1-p_{i,t}\right) +  S_{i}(p_{t})p_{i,t}, \label{Eq:model2}
\end{equation}
where $ p_{i,0} =  X_{i,0} $. It is known that for suitable sets $ \mathcal{H} \subset \mathbb{R}^{n} $ and assuming all the colonisation and survival 
functions are influenced by a large number of sites, 
\[
\sup_{h \in \mathcal{H}} \left| n^{-1} \sum_{i=1}^{n} h_{i}(X_{i,t} - p_{i,t}) \right|
 \]
is small in probability when $ n $ is large \cite{BMP:15,HMP:20}.

Demonstrating the closeness of paths is not the only way to relate stochastic and deterministic models. Allen \cite{Allen:08} (see also \cite{SK:13}) showed 
that the expectation of the stochastic logistic model is underestimated by its deterministic counterpart. A similar result has been demonstrated for the SIR 
epidemic model \cite{WBS:16} and a general non-Markovian network based SIR model \cite{WBS:17}. Our first aim is to establish conditions under which 
the analogous result for model (\ref{Eq:model1}) holds,  namely $ \mathbb{E}_{0} X_{i,t} \leq p_{i,t} $, where $ \mathbb{E}_{0} $ denotes expectation 
conditioned on the initial state $ X_{0} $.

We then consider another type of approximation to (\ref{Eq:model1}) called the independent site approximation. Define $ W_{t} =(W_{1,t},\ldots, W_{n,t}) $ 
where the $ W_{i,\cdot} $ are independent Markov chains on $ \{0,1\} $ such that
\begin{equation}
\mathbb{P}\left(W_{i,t+1} = 1 \middle| W_{i,t} \right) = C_{i}(p_{t}) (1-W_{i,t}) + S_{i}(p_{t})) W_{i,t}, \label{Eq:model3}
\end{equation}
$ W_{i,0} = X_{i,0} $ and $ p_{t} $ satisfies (\ref{Eq:model2}). By construction $ p_{i,t} = \mathbb{E}_{0}(W_{i,t}) $ for all $i $ and all $ t $. The 
independent site approximation is motivated by propagation of chaos type results where finite collections of particles in interacting particle systems evolve 
almost independently of one another under certain conditions \cite{Sznitman:91}. This phenomenon has been demonstrate for a number of population 
models that exhibit a law of large numbers \cite{BL:15,FP:21}. If for a fixed $ i$ and $ t $ the inequality $ \mathbb{E}_{0} X_{i,t} \leq p_{i,t} $ holds, 
then $ X_{i,t} \leq_{\text{st}} W_{i,t}$, where $ \leq_{\text{st}} $ denotes the usual stochastic ordering. Our second aim is to show that (\ref{Eq:model1}) 
is smaller than (\ref{Eq:model3}) in a form of multivariate stochastic ordering called the lower orthant order. This result will not require the process to 
display any law of large numbers behaviour for the process.

As occupancy processes are natural time discretisation for finite spin systems, we obtain analogous results for spin systems. The bound on the expectations 
is obtained using the positive correlations property of spin systems. The stochastic ordering result for spin systems is obtained by applying a limiting 
argument to the occupancy process.

\section{The deterministic system bounds the probability of occupation}

In this section we show the deterministic process (\ref{Eq:model2}) provides a bound on the expected state of the occupancy process (\ref{Eq:model1}). 
The main step in the proof is the application of the Harris inequality.

\begin{theorem} \label{thm1}
Assume that for each $ i$ the functions $ C_{i} $ and $ S_{i} $ extended to $ [0,1]^{n}$  are increasing and concave, and the $ S_{i} - C_{i} $ are 
decreasing and non-negative. If $ p_{i,0} = X_{i,0} $ for all $ i $, then $ \mathbb{E}_{0} X_{i,t} \leq p_{i,t} $ for all $ i $ and all $ t \geq 0 $. 
\end{theorem}

\begin{proof} We can express the Markov chain $X_{t}$ as
\begin{equation} \label{Eq:Proof1}
\begin{split}
X_{i,t+1} &=  ( 1- X_{i,t}) \mathbb{I}(U_{i,t+1} \leq C_{i}(X_{t})) + X_{i,t} \mathbb{I}(U_{i,t+1} \leq S_{i}(X_{t}))  \\
& = \mathbb{I}(U_{i,t+1} \leq C_{i}(X_{t}))  + X_{i,t} \mathbb{I}( C_{i}(X_{t}) \leq U_{i,t+1} \leq S_{i}(X_{t})) =: \mathcal{X}_{i}(U_{i,t+1}, X_{t}),
\end{split}
\end{equation}
where the $ U_{i,t} $ form an array of independent standard uniform random variables. The function $ \mathcal{X}_{i}(u,x) $ is decreasing in $ u $ for 
fixed $ x \in \{0,1\}^{n} $. Also, for any $x,y \in \{0,1\}^{n} $ such that $ x \leq y $ in the partial ordering on $ \{0,1\}^{n} $ (that is, $ x \leq y \iff x_{i} \leq y_{i} $ 
for all $ i $), we have $ \mathcal{X}_{i}(u,x) \leq \mathcal{X}_{i}(u,y) $ for any $ u \in [0,1] $ as $ C_{i}(x) $ and $ S_{i}(x) $ are increasing in $x$. Hence $ X_{i,t} $ 
is a decreasing function of the array $ \{ U_{i,t}\} $. As $ S_{i}(x) - C_{i}(x) $ is decreasing in $ x $, we see $ S_{i}(X_{t}) - C_{i}(X_{t}) $ is an increasing 
function of the array $ \{ U_{i,t}\}$. Taking conditional expectations
\begin{align*}
 \mathbb{E} (X_{i,t+1} \mid X_{t}) & = C_{i}(X_{t}) + \left( S_{i}(X_{t}) - C_{i}(X_{t})\right)\,  X_{i,t} .
\end{align*}
Then taking expectations and applying the Harris inequality
\begin{align*}
\mathbb{E}_{0} X_{i,t+1} & \leq \mathbb{E}_{0} C_{i}(X_{t}) + \mathbb{E}_{0}\left( S_{i}(X_{t}) - C_{i}(X_{t})\right)\,  \mathbb{E}_{0}X_{i,t} \\
&=  (1 - \mathbb{E}_{0}X_{i,t})\, \mathbb{E}_{0} C_{i}(X_{t}) + \mathbb{E}_{0}X_{i,t}\, \mathbb{E}_{0}\left( S_{i}(X_{t}) \right). 
\end{align*}
As $ C_{i} $ and $ S_{i}$ are concave, we can apply Jensen's inequality to obtain
\begin{equation} \label{Eq:ProofIneq}
\mathbb{E}_{0} X_{i,t+1}  \leq (1 - \mathbb{E}_{0}X_{i,t})\,  C_{i}(\mathbb{E}_{0} X_{t}) + \mathbb{E}_{0}X_{i,t}\, S_{i}(\mathbb{E}_{0} X_{t}).
\end{equation}
Write $ \pi_{i,t} = \mathbb{E}_{0}X_{i,t} $. Suppose $ \pi_{i,t} \leq p_{i,t} $ for all $ i $, where $ p_{t} $ satisfies the recursion (\ref{Eq:model2}) with $ p_{i,0} = X_{i,0} $. Then
\begin{align*}
p_{i,t+1} & 
= C_{i}(p_{t}) + (S_{i}(p_{t}) - C_{i}(p_{t})) p_{i,t} \\
& \geq 
C_{i}(p_{t}) + (S_{i}(p_{t}) - C_{i}(p_{t})) \pi_{i,t} = (1-\pi_{i,t}) C_{i}(p_{t}) + S_{i}(p_{t}) \pi_{i,t},
\end{align*}
as $ S_{i} - C_{i} \geq 0 $ and $ \pi_{i,t} \leq p_{i,t} $. Since $ C_{i} $ and $ S_{i}$ are increasing,
\begin{align*}
p_{i,t+1} & \geq (1-\pi_{i,t}) C_{i}(\pi_{t}) +S_{i}(\pi_{t}) \pi_{i,t}  \geq \pi_{i,t+1}.
\end{align*}
Hence, $ \pi_{i,t} \leq p_{i,t} $ for all $ i $ and all $ t \geq 0 $. 
\end{proof}

The deterministic process (\ref{Eq:model2}) requires the functions $ C_{i} $ and $ S_{i} $ to be extended from $ \{0,1\}^{n} $ to $ [0,1]^{n}$. Without 
imposing additional restrictions, these functions do not have a unique extension, but some extensions will be better than others in terms of how close $ p_{i,t} $ 
is to $ \mathbb{E}_{0} X_{i,t}$. Let $ \tilde{p}_{t} $ be the solution to (\ref{Eq:model2}) with the functions $ C_{i} $ and $ S_{i}$ replaced by $ \widetilde{C}_{i} $ 
and $ \widetilde{S}_{i}$ satisfying $ C_{i}(p) \leq \widetilde{C}_{i}(p) $ and $ S_{i}(p) \leq \widetilde{S}_{i}(p) $ for all $ p \in [0,1]^{n} $. If $ p_{t} \leq \tilde{p}_{t} $
 in the partial order on $ [0,1]^{n} $, then 
\[
p_{i,t+1} = (1-p_{i,t}) C_{i}(p_{t}) + p_{i,t} S_{i}(p_{t})  \leq (1-\tilde{p}_{i}) \widetilde{C}_{i}(\tilde{p}_{t}) + \tilde{p}_{i,t} \widetilde{S}_{i}(\tilde{p}_{t}) = \tilde{p}_{i,t+1}.
\]
In light of Theorem \ref{thm1} we prefer smaller extensions of $ C_{i} $ and $ S_{i} $ that are increasing and concave. Methods for constructing the smallest 
concave extension are discussed in \cite{TRX:13}, though the gains achieved with these methods are unlikely to repay the computational effort required for their 
calculation. A relatively simple improvement can be obtained by noting that the occupancy process is not affected by the value assigned to $ C_{i}(x) $ when 
$ x_{i} = 1 $.  Suppose $ C_{i} $ is an increasing concave extension  and define $ \bar{C}_{i}(p) = C_{i}(\tilde{p}) $, where $ \tilde{p}_{j} = p_{j} $ for 
$ j \neq i $ and $ \tilde{p}_{i} = 0 $. The function $ \bar{C}_{i} $ is an increasing concave function which satisfies $ C_{i}(p) \geq \bar{C}_{i}(p) $ for all 
$ p \in [0,1]^{n} $. This means we should avoid extensions of $ C_{i} $ which result in the deterministic process being `self-colonising', that is the value of 
$ p_{i} $ affecting the value of $ C_{i}(p) $.
Similar comments apply to the extension of $ S_{i} $ since the process is not affected by the value of $ S_{i}(x) $ when $ x_{i} = 0 $.

Since occupancy processes can be viewed as a time discretisation of finite spin systems \cite[Algorithms 1 \& 2]{MH:21}, it is natural consider a version of 
Theorem \ref{thm1} for those processes. Any finite spin system $ (X,\, t \in \mathbb{R}_{+}) $ can be represented as a Markov jump process in the usual transition 
notation:
\begin{equation} \label{Eq:model4}
\begin{aligned}
X_i:\quad 0\to 1 &\quad \mbox{at rate } \lambda_i(X)\\
1\to 0 &\quad \mbox{at rate } \mu_i(X)
\end{aligned} \qquad \mbox{for }i=1,\dots,n,
\end{equation}
where $ \lambda_{i},\, \mu_{i} :\{0,1\}^{n} \to \mathbb{R}_{+} $. The expectation of $ X_{i,t} $ satisifies
\begin{equation} \label{Eq10}
\mathbb{E}_{0} X_{i,t} =  X_{i,0} + \int^{t}_{0} \mathbb{E}_{0}\left(  (1 - X_{i,s}) \lambda_{i}(X_{s}) - X_{i,s} \mu_{i}(X_{s}) \right) ds. 
\end{equation}
Provided the functions $ \lambda_{i} $ and $ \mu_{i} $ can be extended from $\{0,1\}^{n} $ to  $ [0,1]^{n}$, this suggests the deterministic approximation 
for the spin system is the solution to the system of ordinary differential equations
\begin{equation} \label{Eq:model5}
p^{\prime}_{i,t} = (1-p_{i,t}) \lambda_{i}(p_{t}) - p_{i,t} \mu_{i}(p_{t}).
\end{equation}

\begin{theorem} \label{Cor1}
Assume that for each $ i$ the functions $ \lambda_{i} $ extended to $ [0,1]^{n} $ are increasing and concave, $ \mu_{i} $ extended to $ [0,1]^{n} $ are 
decreasing and convex, and the $ \lambda_{i} +\mu_{i} $ are increasing. If $ p_{i,0} = X_{i,0} $ for all $ i $, then $ \mathbb{E}_{0} X_{i,t} \leq p_{i,t} $ for all 
$ i $ and all $ t \geq 0 $. 
\end{theorem}

\begin{proof}
With the  $ \lambda_{i} $ increasing and the $ \mu_{i} $ decreasing, the spin system (\ref{Eq:model4}) is said to be attractive \cite[III Defintion 2.1]{Liggett:05}. 
An attractive spin system $X$ with fixed initial condition $ X_{0} $ has positive correlations at all times $ t \geq 0 $, that is
\[
\mathbb{E}(f(X_{t}) g(X_{t})) \geq \mathbb{E} f(X_{t}) \mathbb{E} g(X_{t}) 
\]
for all continuous functions $ f $ and $ g $ that are monotone in the sense $ f(\eta) \leq f(\zeta) $ whenever $ \eta \leq \zeta $ \cite[II Theorem 2.14, III Theorem 2.2]{Liggett:05}. 
Let $ \pi_{i,t} = \mathbb{E}_{0} X_{i,t} $.  Differentiating (\ref{Eq10}) gives
\begin{equation}
\pi_{i,t}^{\prime} = \mathbb{E}_{0}\left(  (1 - X_{i,t}) \lambda_{i}(X_{t}) - X_{i,t} \mu_{i}(X_{t}) \right). \label{Eq10b}
\end{equation}
As $ \lambda_{i}(\cdot) + \mu_{i}(\cdot) $ is increasing, we can apply the positive correlations property to (\ref{Eq10b}) to obtain
\begin{align*}
\pi_{i,t}^{\prime} & \leq  \mathbb{E}_{0}  \lambda_{i}(X_{t}) - \mathbb{E}_{0} X_{i,t} \mathbb{E}_{0} (\lambda_{i}(X_{t}) + \mu_{i}(X_{t}))  \\
& \leq   (1 - \mathbb{E}_{0}X_{i,t} ) \mathbb{E}_{0}  \lambda_{i}(X_{t}) - \mathbb{E}_{0} X_{i,t} \mathbb{E}_{0} \mu_{i}(X_{t}).
\end{align*}
As $ \lambda_{i} $ is concave and $ \mu_{i} $ is convex, Jensen's inequality yields
\begin{align}
\pi_{i,t}^{\prime} & \leq (1 - \mathbb{E}_{0} X_{i,s} )   \lambda_{i}(\mathbb{E}_{0} X_{s}) - \mathbb{E}_{0} X_{i,s}  \mu_{i}(\mathbb{E}_{0} X_{s})  = (1-\pi_{i,t}) \lambda_{i}(\pi_{t}) - \pi_{i,t} \mu_{i}(\pi_{t}) \label{Eq11}
\end{align}
Define the functions $\phi_{i}: [0,1]^{n} \to \mathbb{R}_{+}$ such that 
\[
\phi_{i}(u_{1},\ldots,u_{n}) := (1-u_{i}) \lambda_{i}(u) - u_{i} \mu_{i}(u).
\]
As the $ \lambda_{i}$ are increasing and the $ \mu_{i}$ are decreasing, each function $\phi_{i}$ is non-decreasing in each $u_{j} $ for $j\neq i$. The
system of differential inequalities (\ref{Eq11}) satisfies the conditions of a result by Wa\.{z}ewski \cite{Wazup} (see also \cite[Theorem 1 of Section 13 in Chapter XI]{Mitrinovic:1991}), 
which allows us to conclude that if $ p_{i,0} = X_{i,0} $ for all $i $, then $ \pi_{i,t} \leq p_{i,t} $ for all $i $ and $ t \geq 0$.
\end{proof}

\section{Propagation of chaos and stochastic ordering}

An interacting particle system is said to display propagation of chaos if the particles  evolve almost independently of one another when the system size is 
large. Demonstrating this behaviour usually involves showing a law of large numbers holds so that the transition rates of the individual particles are well 
approximated by some deterministic process. A propagation of chaos result was established for the occupancy process in \cite{BMP:15,HMP:20}, where the 
independent site approximation was coupled to the occupancy process and the two processes shown to be close over finite time intervals. 

Instead of attempting to show the occupancy process is close to the independent site approximation, in this section we show that the occupancy process is 
dominated by the independent site approximation in a certain sense. 

A weak notion of multivariate stochastic ordering is the lower orthant order \cite[Section 6.G.1]{SS:07}. We say that the random vector $ \mathbf{Y} $ is 
smaller than the random vector $ \mathbf{Z} $ in the lower orthant order, denoted $ \mathbf{Y} \leq_{\text{lo}} \mathbf{Z} $, if
\[
\mathbb{P}(Y_{1} \leq \zeta_{1}, \ldots, Y_{m} \leq \zeta_{m} ) \geq \mathbb{P}(Z_{1} \leq \zeta_{1},\ldots, Z_{m} \leq \zeta_{m})
\]
for all $ (\zeta_{1},\ldots, \zeta_{m}) \in \mathbb{R}^{m} $. For distributions on the hypercube $ \{0,1\}^{m}$, this condition reduces to 
\[
\mathbb{P}(Y_{i} = 0 \text{ for all } i \in A) \geq \mathbb{P}(Z_{i} = 0 \text{ for all } i \in A),
\]
for all subsets $ A \subset \{1,2,\ldots,m\}$. Write $ \mathbb{P}_{0} $ to denote conditioning on the initial state $ X_{0} $. The Harris inequality applied 
to the construction (\ref{Eq:Proof1}) shows 
\[
\mathbb{P}_{0}\left(X_{i,t} = 0 \text{ for all } i \in A \right) = \mathbb{E}_{0} \left( \prod_{i \in A} (1 - X_{i,t}) \right) \geq \prod_{i\in A} \mathbb{E}_{0} (1 - X_{i,t}),
\]
for a given $t $ and all subsets $ A \subset \{1,2,\ldots,n\}$. Then applying Theorem 2.1 we see
\[
\prod_{i\in A} \mathbb{E}_{0} (1 - X_{i,t})  \geq  \prod_{i \in A} \mathbb{E}_{0}(1 - W_{i,t})  = \mathbb{P}_{0}\left(W_{i,t} = 0 \text{ for all } i \in A \right).
\]
This establishes $ X_{t} \leq_{\text{lo}} W_{t} $ for a given time $ t \geq 0$. We would like to establish the ordering relation between $ X$ and $ W $ for all 
times in the sense that for any subset $ A \subseteq \{1,\ldots,n\} $, positive integers $ m_{i} $ and times $ t_{i,1},\ldots, t_{i,m_{i}} $
\begin{equation} \label{Eq:PC1}
\begin{aligned}
\mathbb{P}_{0}\left(X_{i,t_{i,j}} = 0 \text{ for all } i \in A,\right.&\left. \ j  \in  \{1,\ldots,m_{i}\} \right)  \\ &  \geq \mathbb{P}_{0}\left(W_{i,t_{i,j}} = 0 \text{ for all } i \in A,\ j \in  \{1,\ldots,m_{i}\} \right).
\end{aligned}
\end{equation}
Note that for each $ i \in A$, the set of times $t_{i,1},\ldots, t_{i,m_{i}} $ may be different.

\begin{theorem} \label{thm2}
Assume the conditions of Theorem \ref{thm1} hold. Assume also that for all $ i $, $ S_{i} - C_{i} $ is convex. The process $ (X, t \in \mathbb{N}) $ given 
by (\ref{Eq:model1}) is smaller in the lower orthant order than the process $ (W, t \in\mathbb{N} ) $ given by (\ref{Eq:model3}).
\end{theorem}

\begin{proof}
Let $ A $ be a subset of $ \{1,\ldots,n\} $, and for each $ i \in A $ take a positive integer $ m_{i} $ and times $ t_{i,1},\ldots, t_{i,m_{i}} $. Then by the Harris inequality
\begin{align*}
\lefteqn{\mathbb{P}_{0}\left(X_{i,t_{i,j}} = 0 \text{ for all } i \in A,\ j  \in  \{1,\ldots,m_{i}\} \right)} \\
 & = \mathbb{E}_{0} \left( \prod_{i \in A} \prod_{j=1}^{m_{i}} \left(1 - X_{i,t_{i,j}} \right) \right)  \geq  \prod_{i \in A} \mathbb{E}_{0} \left( \prod_{j=1}^{m_{i}} \left(1 - X_{i,t_{i,j}} \right) \right) \\
& \geq \prod_{i \in A} \mathbb{P}_{0} \left(X_{i,t_{i,j}} = 0 \text{ for all }  j \in \{1,\ldots,m_{i}\} \right).
\end{align*}

It remains to show that for each $ i $
\[
\begin{aligned}
\mathbb{P}_{0}\left(X_{i,t_{i,j}} = 0 \text{ for all }  j \in \{1,\ldots,m_{i}\} \right) \geq \mathbb{P}_{0}\left(W_{i,t_{i,j}} = 0 \text{ for all }  j \in \{1,\ldots,m_{i}\} \right).
\end{aligned}
\]
Now for   $ \boldsymbol{\omega}= (\omega_{1},\ldots,\omega_{m}) \in \{0,1\}^{m}$ define
\begin{align*}
P^{X}_{m}(\boldsymbol{\omega}) & := \mathbb{P}_{0}\left(X_{i,1} \leq \omega_{1}, \ldots, X_{i,m} \leq \omega_{m}\right) = \mathbb{E}_{0} \left[\prod_{t=1}^{m} (1 - X_{i,t})^{1-\omega_{t}} \right].
\end{align*}
and define $ P^{W}_{m}(\boldsymbol{\omega}) $ similarly. We prove by induction that $ P^{X}_{m} (\boldsymbol{\omega}) \geq P^{W}_{m}(\boldsymbol{\omega}) $ for all $ \boldsymbol{\omega} \in \{0,1\}^{m}$ and all $ m \geq 1 $.

Assume $ \boldsymbol{\omega} \neq \mathbf{1} $ and let $ \phi(\boldsymbol{\omega}) = \max\{j: \omega_{j} = 0 \} $. If $ \phi(\boldsymbol{\omega}) = 1 $, then from Theorem \ref{thm1}
\[
P^{X}_{m}(\omega_{1},\ldots,\omega_{m}) = \mathbb{E}_{0} (1-X_{i,1}) \geq \mathbb{E}_{0} (1- W_{i,1}) = P^{W}_{m}(\omega_{1},\ldots,\omega_{m}).
\]
Suppose now that $\phi(\boldsymbol{\omega}) = \tilde{m} \geq 2$. Then 
\begin{align*}
\lefteqn{P^{X}_{m}(\boldsymbol{\omega})} \\
 & =  \mathbb{E}_{0} \left[\prod_{t=1}^{\tilde{m}} (1 - X_{i,t})^{1-\omega_{t}} \right] =   \mathbb{E}_{0}\left[\mathbb{E}_{0}\left[(1-X_{i,\tilde{m}}) \middle| X_{\tilde{m}-1}\right]\prod_{t=1}^{\tilde{m}-1} (1 - X_{i,t})^{1-\omega_{t}} \right]\\
 & = \mathbb{E}_{0} \left[\left((1-S_{i}(X_{\tilde{m}-1})) + (1-X_{i,\tilde{m}-1})(S_{i}(X_{\tilde{m}-1})-C_{i}(X_{\tilde{m}-1}))\right) \prod_{t=1}^{\tilde{m}-1}(1-X_{i,t})^{1-\omega_{t}} \right] .
\end{align*}
The function $ S_{i}(x) - C_{i}(x)$ is decreasing in $x$ by assumption and $X_{i,t}$ is a decreasing function of the array $\{U_{i,t}\}$ by construction 
(\ref{Eq:Proof1}). As it is composition of two decreasing functions, $ S_{i}(X_t) - C_{i}(X_t) $ is an increasing function of the array $\{ U_{i,t} \}$. Applying 
the Harris inequality  shows
\begin{multline*}
\mathbb{E}_{0} \left[(1-X_{i,\tilde{m}-1})(S_{i}(X_{\tilde{m}-1})-C_{i}(X_{\tilde{m}-1})\prod_{t=1}^{\tilde{m}-1}(1-X_{i,t})^{1-\omega_{t}} \right] \\
 \geq \mathbb{E}_{0} \left[ (S_{i}(X_{\tilde{m}-1}) - C_{i}(X_{\tilde{m}-1}) \right] P^{X}_{m}(\omega_{1},\ldots,\omega_{\tilde{m}-2},0,1,\ldots,1). 
\end{multline*}
Since $ S_{i} - C_{i} $ is also convex, Jensen's inequality with Theorem \ref{thm1} shows
\begin{multline*}
\mathbb{E}_{0} \left[(1-X_{i,\tilde{m}-1})(S_{i}(X_{\tilde{m}-1})-C_{i}(X_{\tilde{m}-1})\prod_{t=1}^{\tilde{m}-1}(1-X_{i,t})^{1-\omega_{t}} \right] \\
 \geq (S_{i}(p_{\tilde{m}-1}) - C_{i}(p_{\tilde{m}-1}) ) P^{X}_{m}(\omega_{1},\ldots,\omega_{\tilde{m}-2},0,1,\ldots,1). 
\end{multline*}
The same argument shows
\begin{align*}
\mathbb{E}_{0} \left[(1-S_{i}(X_{\tilde{m}-1})) \prod_{t=1}^{\tilde{m}-1}(1-X_{i,t})^{1-\omega_{t}} \right] &\geq (1-S_{i}(p_{\tilde{m}-1})) P^{X}_{m}(\omega_{1},\ldots,\omega_{\tilde{m}-1},1,\ldots,1).
\end{align*}
Therefore,
\begin{multline}
P^{X}_{m}(\boldsymbol{\omega})
 \geq (1-S_{i}(p_{\tilde{m}-1})) P^{X}_{m}(\omega_{1},\ldots,\omega_{\tilde{m}-1},1,\ldots,1) \\+ \left( (S_{i}(p_{\tilde{m}-1}) - C_{i}(p_{\tilde{m}-1}) \right) P^{X}_{m}(\omega_{1},\ldots,\omega_{\tilde{m}-2},0,1,\ldots,1). \label{Eq:Proof3}
\end{multline}
On the other hand, for the process $ W $ given in (\ref{Eq:model3})
\begin{align*}
P^{W}_{m}(\boldsymbol{\omega})
 & =  \mathbb{E}_{0} \left[\prod_{t=1}^{\tilde{m}} (1 - W_{i,t})^{1-\omega_{t}} \right] =   \mathbb{E}_{0} \left[\mathbb{E}_{0} \left[(1-W_{i,\tilde{m}}) \middle| W_{\tilde{m}-1}\right]\prod_{t=1}^{\tilde{m}-1} (1 - W_{i,t})^{1-\omega_{t}} \right]\\
 & = \mathbb{E}_{0}  \left[\left((1-S_{i}(p_{\tilde{m}-1})) + (1-W_{i,\tilde{m}-1})(S_{i}(p_{\tilde{m}-1})-C_{i}(p_{\tilde{m}-1})\right) \prod_{t=1}^{\tilde{m}-1}(1-W_{i,t})^{1-\omega_{t}} \right] .
\end{align*}
Therefore,
\begin{multline}
P^{W}_{m}(\boldsymbol{\omega})
 = (1-S_{i}(p_{\tilde{m}-1})) P^{W}_{m}(\omega_{1},\ldots,\omega_{\tilde{m}-1},1,\ldots,1) \\+ \left( (S_{i}(p_{\tilde{m}-1}) - C_{i}(p_{\tilde{m}-1}) \right) P^{W}_{m}(\omega_{1},\ldots,\omega_{\tilde{m}-2},0,1,\ldots,1).\label{Eq:Proof4}
\end{multline}
If $ P^{X}_{m} (\boldsymbol{\omega}) \geq P^{W}_{m}(\boldsymbol{\omega}) $ for all $ \boldsymbol{\omega} \in \{0,1\}^{m}$ such that 
$ \phi(\boldsymbol{\omega}) \leq \tilde{m}-1 $, then comparing (\ref{Eq:Proof3}) and (\ref{Eq:Proof4}) shows $ P^{X}_{m} (\boldsymbol{\omega}) \geq P^{W}_{m}(\boldsymbol{\omega}) $ 
for all $ \boldsymbol{\omega} \in \{0,1\}^{m}$ such that $ \phi(\boldsymbol{\omega}) \leq \tilde{m} $.
\end{proof}

For spin systems the independent site approximation is given by $ W_{t} =(W_{1,t},\ldots, W_{n,t}) $ where the $ W_{i,\cdot} $ are independent 
Markov chains on $ \{0,1\} $ such that
\begin{equation} \label{Eq:model7}
\begin{aligned}
W_i:\quad 0\to 1 &\quad \mbox{at rate } \lambda_i(p_{t})\\
1\to 0 &\quad \mbox{at rate } \mu_i(p_{t})
\end{aligned} \qquad \mbox{for }i=1,\dots,n,
\end{equation}
$ W_{i,0} = X_{i,0} $ and $ p_{t} $ satisfies (\ref{Eq:model5}). Since the lower orthant order is closed under convergence in distribution 
\cite[Theorem 6.G.3(d)]{SS:07}, a limiting argument can be used to prove the following result.

\begin{theorem}
Assume the conditions of Theorem \ref{Cor1} hold. Assume also that for all $ i $, $ \lambda_{i} + \mu_{i} $ is concave, and each of the 
$ \lambda_{i}$ and $\mu_{i}$ are Lipschitz continuous. The process $ (X, t \in \mathbb{R}_{+}) $ given by (\ref{Eq:model4}) is smaller in the 
lower orthant order (\ref{Eq:PC1}) than the process $ (W, t \in\mathbb{R}_{+} ) $ given by (\ref{Eq:model7}).
\end{theorem}

\begin{proof}
For $ \delta > 0 $ sufficiently small, let $ (X^{\delta}, t \in \mathbb{N}) $ be the occupancy process with 
\[
C_{i}(x) = \delta \lambda_{i} (x), \quad \text{and} \quad S_{i}(x) = 1 - \delta \mu_{i}(x).
\]
The assumptions of Theorem \ref{thm2} are satisfied by $ X^{\delta}$. Let $(W^{\delta}, t \in \mathbb{N})$ be the corresponding independent 
site approximation (\ref{Eq:model3}) so $ X^{\delta} $ is smaller than $W^{\delta} $ in the lower orthant order. Let $ (N, t> 0)$ be a unit rate 
Poisson process independent of $ X^{\delta} $ and $W^{\delta}$. Define the continuous time process $ (\widetilde{X}^{\delta}, t \in \mathbb{R}_{+}) $ 
by $ \widetilde{X}^{\delta}_{t} := X^{\delta}_{N(\delta^{-1}t)}$ and $ (\widetilde{W}^{\delta}, t \in \mathbb{R}_{+}) $ by 
$ \widetilde{W}^{\delta}_{t} := W^{\delta}_{N(\delta^{-1}t)}$. Then $\widetilde{X}^{\delta} $ is smaller than $ \widetilde{W}^{\delta}$ in the 
lower orthant order as the lower orthant order is closed under mixtures \cite[Theorem 6.G.3 (e)]{SS:07}. It remains to show $ \widetilde{X}^{\delta} \stackrel{d}{\to} X$ 
and $\widetilde{W}^{\delta} \stackrel{d}{\to} W$ since the lower orthant order is preserved under convergence in distribution \cite[Theorem 6.G.3 (d)]{SS:07}. 

From the uniformization construction \cite[Section 2.1]{Keilson:79}, the process $X^{\delta}$ is a continuous time Markov chain on $\{0,1\}^{n}$ with transition rates:
\[
\begin{split}
q^{\delta}_{X} (x,y) = \delta^{-1} \prod_{i=1}^{n}\left[ \left( \delta \lambda_{i}(x) \right)^{(1-x_{i})(y_{i} - x_{i})_{+}} \right. & \left( 1 - \delta \lambda_{i}(x) \right)^{(1-x_{i})(1-(y_{i} - x_{i})_{+})}   \\
& \times \left. \left( \delta \mu_{i}(x) \right)^{x_{i}(x_{i} -y_{i})_{+}} \left( 1 - \delta \mu_{i}(x) \right)^{x_{i}(1- (x_{i} -y_{i})_{+})} \right], 
\end{split}
\]
for any $ x,\ y \in \{0,1\}^{n}$. As $ \delta \to 0$, $q^{\delta}_{X} $ converges to the transition rates of (\ref{Eq:model4}) and since the 
state space is finite, this is sufficient to show $\widetilde{X}^{\delta} \stackrel{d}{\to} X$ .

The process $ \left(N(\delta^{-1}t), \widetilde{W}^{\delta}_t \right)$ is a continuous time Markov chain on $ \mathbb{N}_{0} \times  \{0,1\}^{n} $ 
with transitions $ (m,w) \to (m+1, w + u) $ for $ u \in \{-1,0,1\}^{n}$ at rate
\[
\begin{split}
\beta_{u}(m,w) = \delta^{-1} \prod_{i=1}^{n} \left[ \left( \delta  \lambda_{i}(p^{\delta}_{m}) \right)^{(1-w_{i})(u_{i})_{+}} \right.  & \left( 1 - \delta \lambda_{i}(p^{\delta}_{m}) \right)^{(1-w_{i})(1-(u_{i})_{+})} \\
& \times \left. \left( \delta \mu_{i}(p^{\delta}_{m}) \right)^{w_{i}(-u_{i})_{+}} \left( 1 - \delta \mu_{i}(p^{\delta}_{m}) \right)^{w_{i}(1- (-u_{i})_{+})} \right], 
\end{split}
\]
where 
\[
p^{\delta}_{i,m+1} = \delta \lambda_{i}(p^{\delta}_{m}) (1 - p_{i,m}^{\delta}) + (1- \delta \mu_{i}(p_{m}^{\delta})) p_{i,m}^{\delta}.
\]
We can represent $ \left(N(\delta^{-1}t), \widetilde{W}^{\delta}_{t} \right)$ as a random time change of Poisson processes \cite[Chapter 6, Section 4]{EK:05} 
\begin{align*}
\widetilde{W}^{\delta}_{t} & = X_{0} + \sum_{u \in \{-1,0,1\}^{n}} u\, N_{u} \left( \int^{t}_{0} \beta_{u} (N(\delta^{-1} s), \widetilde{W}^{\delta}_{s})\, ds \right) \\
N(\delta^{-1}t) & = \sum_{u \in \{-1,0,1\}^{n}} N_{u} \left( \int^{t}_{0} \beta_{u} (N(\delta^{-1} s), \widetilde{W}^{\delta}_{s})\, ds \right),
\end{align*}
where the $N_{u} $ are independent unit rate Poisson processes. The process $ (W , t \in \mathbb{R}_{+})$ can be constructed on the same probability 
space as $ \left(N(\delta^{-1}t), \widetilde{W}^{\delta}_{t} \right)$ by representing $(W, t \in \mathbb{R}_{+})$ as
\[
W_{i,t} = X_{i,0} + N_{i}^{+}\left( \int^{t}_{0} (1 - W_{i,s})\, \lambda_{i}(p_{s})\, ds \right) - N_{i}^{-}\left( \int^{t}_{0} W_{i,s}\, \mu_{i}(p_{s})\, ds \right),
\]
 where $ p_{t}$ is the solution to (\ref{Eq:model5}) and identifying $N_{i}^{+}$ and $ N_{i}^{-}$ with the unit rate Poisson processes $N_{u}$ such 
 that $ u_{i} = \pm 1 $ and $ u_{j} = 0 $ for all $ j \neq i$. We now use Gronwall's inequality to show $ \mathbb{E} |W_{i,t} - \widetilde{W}^{\delta}_{i,t} | \to 0 $ as $ \delta \to 0$ 
 for all $ t$ and all $ i$, hence $ \widetilde{W}^{\delta} \stackrel{d}{\to} W $. By the triangle inequality,
\begin{align}
    \mathbb{E} |W_{i,t} - \widetilde{W}^{\delta}_{i,t} | & \leq \mathbb{E} \left|\int^{t}_{0} (1-W_{i,s}) \lambda_{i}(p_{s})\, ds - \int^{t}_{0} (1-\widetilde{W}^{\delta}_{i,s}) \lambda_{i}(p^{\delta}_{N(\delta^{-1}s)})\, ds \right| \nonumber \\
    & + \mathbb{E} \left|\int^{t}_{0} W_{i,s} \mu_{i}(p_{s})\, ds - \int^{t}_{0} \widetilde{W}^{\delta}_{i,s} \mu_{i}(p^{\delta}_{N(\delta^{-1}s)})\, ds \right| \nonumber \\
    & + \sum_{u: \|u\| \geq 2} \mathbb{E} \int^{t}_{0}  \beta_{u}(N(\delta^{-1} s), \widetilde{W}^{\delta} (s))\, ds, \label{Thm4:Eq1}
\end{align}
where $\|u\| = \sum_{i=1}^{n} |u_{i}|$. As the $\lambda_{i} $ and  $\mu_{i}$ are Lipschitz continuous, there exists constants $ C_{1} $ and $ C_{2} $ such that 
\begin{align}
      \lefteqn{\mathbb{E} \left|\int^{t}_{0} \left(1- W_{i,s}\right) \lambda_{i}(p_{s})\, ds - \int^{t}_{0} \left(1 - \widetilde{W}^{\delta}_{i,s}\right) \lambda_{i}(p^{\delta}_{N(\delta^{-1}s)})\, ds \right|} \nonumber \\
       & + \mathbb{E} \left|\int^{t}_{0} W_{i,s} \mu_{i}(p_{s})\, ds - \int^{t}_{0} \widetilde{W}^{\delta}_{i,s} \mu_{i}(p^{\delta}_{N(\delta^{-1}s)})\, ds \right| \nonumber \\
     & \leq C_{1} \int^{t}_{0} \mathbb{E} |W_{i,s} - \widetilde{W}^{\delta}_{i,s} |\, ds   + C_{2} \int^{t}_{0} \mathbb{E} \| p^{\delta}_{N(\delta^{-1}s)} - p_{s} \| \, ds. \label{Thm4:Eq2}
\end{align}
The usual argument for proving convergence of Euler's method \cite[Theorem 212A]{Butcher:16} shows that for any $ t \geq 0 $, there exists a constant $C_{3}$ such that 
\[
\| p^{\delta}_{N(\delta^{-1}t)} - p_{t} \| \leq \frac{e^{C_{3} t} -1 }{C_{3}} \left| \delta N(\delta^{-1} t) - t\right|
\]
so 
\begin{equation} \label{Thm4:Eq3}
\mathbb{E} \| p^{\delta}_{N(\delta^{-1}t)} - p_{t} \|  \leq \delta t \frac{e^{C_{3} t} -1 }{C_{3}}. 
\end{equation}
For any $ u$ such that $ \sum_{i=1}^{n} |u_{i}| \geq 2$, there exists a constant $C_{4}$ such that 
\begin{equation}\label{Thm4:Eq4}
\beta_{u}(N(\delta^{-1} t), \widetilde{W}^{\delta} (t)) \leq C_{4} \delta.
\end{equation}
Combining $(\ref{Thm4:Eq1}) - (\ref{Thm4:Eq4})$ and applying Gronwall's inequality, we see that $ \mathbb{E} |W_{i,t} - \widetilde{W}^{\delta}_{i,t} | \to 0 $ as $ \delta \to 0$ for all $ t$ and all $ i$.

\end{proof}




\providecommand{\bysame}{\leavevmode\hbox to3em{\hrulefill}\thinspace}
\providecommand{\MR}{\relax\ifhmode\unskip\space\fi MR }
\providecommand{\MRhref}[2]{%
  \href{http://www.ams.org/mathscinet-getitem?mr=#1}{#2}
}
\providecommand{\href}[2]{#2}

\begin{acks}
We would like to thank the two reviewers for their helpful comments and suggestions.
\end{acks}


\end{document}